\documentclass[11pt]{amsart}
\usepackage[english]{babel}
\usepackage{inputenc}

\usepackage[subpreambles=true]{standalone}
\usepackage{import}

\usepackage{latexsym}
\usepackage{amssymb}
\usepackage{amsmath}
\usepackage{amsfonts}
\usepackage[mathscr]{eucal}
\usepackage{xcolor}

\usepackage{amscd}

\usepackage{graphicx}
\usepackage{epsfig}
\usepackage{relsize}

\usepackage{fullpage}

\usepackage{longtable}
\usepackage{rotating}
\usepackage{array}
\usepackage{multicol}
\usepackage{multirow}
\usepackage{makecell}
\usepackage{mathtools}
\usepackage{booktabs}

\usepackage{amsthm}
\usepackage[all,cmtip]{xy}
\usepackage[pdfstartview=FitH]{hyperref}
\usepackage{comment}
\usepackage{psfrag}
\usepackage{verbatim}
\usepackage{hyperref}
\hypersetup{
colorlinks=true,
linkcolor=blue
}
\usepackage{pifont}
\usepackage{caption}
\usepackage{csquotes}

\usepackage{soul}
 \usepackage[normalem]{ulem}

\usepackage{tikz}
\usepackage{pgfplots}
\pgfplotsset{compat=1.15}
\usepackage{mathrsfs}
\usetikzlibrary{arrows}

\usepackage[style=numeric,backend=biber,maxbibnames=99]{biblatex}

\newtheorem{thm}{Theorem}[section]
\newtheorem*{thm*}{Theorem}
\newtheorem{cor}[thm]{Corollary}	
\newtheorem*{cor*}{Corollary}		
\newtheorem{lem}[thm]{Lemma}
\newtheorem{prop}[thm]{Proposition}

\newtheorem{theorem}{Theorem}
\renewcommand*{\thetheorem}{\Alph{theorem}}

\theoremstyle{definition}
\newtheorem{definition}[thm]{Definition} 	
\newtheorem*{rema}{Remark} 
\newtheorem*{definition*}{Definition}    	

\numberwithin{equation}{section}

\DeclareMathOperator{\dist}{dist}

\DeclareMathOperator{\vol}{vol}

\newcommand{\norm}[1]{\left\Vert#1\right\Vert}

\begin{document}



\title[RCSP]{The \textit{rival} coffee shop problem}


\author[J.~Casado]{Javier Casado$^{\ast}$}

\author[M.~Cuerno]{Manuel Cuerno$^{\ast\ast}$}


\thanks{$^*$Supported in part by the FPU Graduate Research Grant FPU20/01444, and by research grants  
	 MTM2017-‐85934-‐C3-‐2-‐P and PID2021-124195NB-C32
from the Ministerio de Econom\'ia y Competitividad de Espa\~{na} (MINECO)} 

\thanks{$^{\ast\ast}$Supported in part by the FPI Graduate Research Grant PRE2018-084109, and by research grants  
	 MTM2017-‐85934-‐C3-‐2-‐P and PID2021-124195NB-C32
from the Ministerio de Econom\'ia y Competitividad de Espa\~{na} (MINECO)}


\address[J.~Casado]{Department of Mathematics, Universidad Aut\'onoma de Madrid and ICMAT CSIC-UAM-UC3M, Spain}
\email{javier.casadoa@uam.es} 

\address[M.~Cuerno]{Department of Mathematics, Universidad Aut\'onoma de Madrid and ICMAT CSIC-UAM-UC3M, Spain}
\email{manuel.mellado@uam.es}


\date{\today}




\begin{abstract}

\textcolor{black}{In this paper, we will address a modification of the following optimization problem: given a positive integer $N$ and a compact Riemannian manifold X, the goal is to place a point $x_N\in X$ in such a way that the sequence $\{x_1,\dots,x_N\}\subset X$ is distributed as uniformly as possible, considering that $\{x_1,\dots,x_{N-1}\}\subset X$ already is. This can be thought as a way of placing coffee shops in a certain area one at a time in order to cover it optimally. So, following this modelization we will denote this problem as the coffee shop problem. This notion of optimal settlement is formalized in the context of optimal transport and Wasserstein distance. As a novel aspect, we introduce a new element to the problem: the presence of a rival brand, which competes against us by opening its own coffee shops.}  As our main tool, we use a variation of the Wasserstein distance (the Signed Wasserstein distance presented by Piccoli, Rossi and Tournus in \cite{piccoli}), that allows us to work with finite signed measures and fits our problem. We present different results \textcolor{black}{depending on} how fast the rival is able to grow. With the Signed Wasserstein distance, we are able to obtain similar inequalities to the ones produced by the canonical Wasserstein one.

\end{abstract}
\setcounter{tocdepth}{1}

\maketitle






\section{Introduction}

\color{black}

Consider the following optimization problem: let $X$ be a compact Riemannian manifold of dimension $d$ and $N$ a positive integer, the aim is to settle $x_N\in X$ such that the sequence $\{x_1,\dots,x_N\}\subset X$ is placed as uniformly distributed as possible, regarding that $\{x_1,\dots,x_{N-1}\}\subset X$ already is.  In real life, if $X$ were a city and we wanted to control $X$ with our brand of coffee shops (each store is modeled by $x_i$), the process would not involve simultaneous openings. Instead, each new store would be introduced one at a time, with careful consideration taking into account the available areas of the region, i.e., we pretend to uniformly cover $X$ at each step. For that reason, from now on, we will denote this problem as the \textit{coffee shop problem}.

\color{black}
\begin{figure}[ht]
\centering
\resizebox{10\columnwidth/28}{!}{%
\begin{tikzpicture}[scale=3]
  \draw (0,0) rectangle (2,2);
  \draw (1,0) -- (1,2);
  \draw (0,1) -- (2,1);
  \foreach \i/\j/\k in {1/1/3, 1/2/1, 2/1/4, 2/2/2} {
    \filldraw (\i-0.5,\j-0.5) circle (1pt) node[below right] {\Large $x_{\k}$};
    }
\end{tikzpicture}}
\caption{}
    \label{figure1}
\end{figure}
\color{black}
 
\color{black}
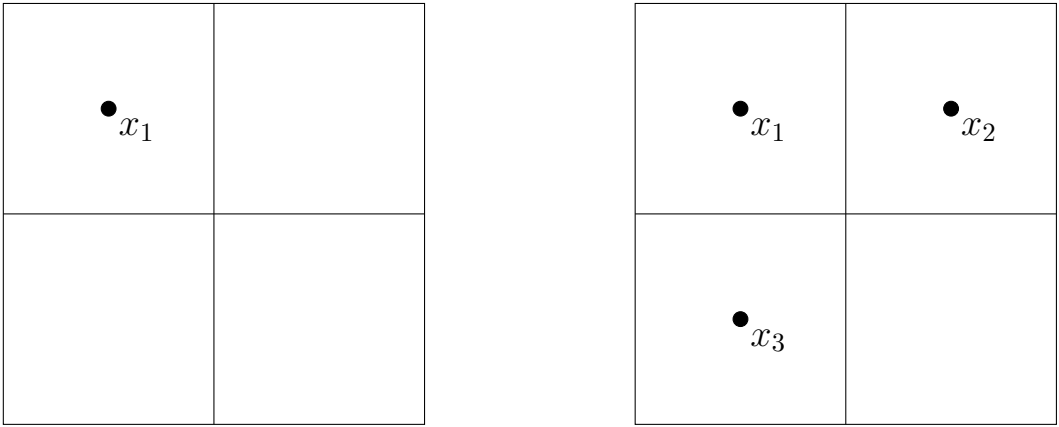
\begin{figure}[ht]
\centering
\begin{tikzpicture}[scale=2.8]
  \draw (0,0) rectangle (2,2);
  \draw (3,0) rectangle (5,2);
  \draw (4,0) -- (4,2);
  \draw (3,1) -- (5,1);
  \draw (1,0) -- (1,2);
  \draw (0,1) -- (2,1);
  \foreach \i/\j/\k in {4/1/3, 4/2/1, 1/2/1, 5/2/2} {
    \filldraw (\i-0.5,\j-0.5) circle (1pt) node[below right] {\Large $x_{\k}$};
    }
\end{tikzpicture}
\caption{On the left: one Coffee shop with the setup of Figure \ref{figure1}. On the right: three Coffee shops with the setup of Figure \ref{figure1} }
    \label{figure2}
\end{figure} 

\color{black}
This question deviates from the task of placing $\{x_1,\dots,x_N\}\subset X$ at once as uniformly distributed as possible. Figure \ref{figure1} shows an example with $N=4$ and $X=[0,1]\times[0,1]$ for this different problem. While examining the coffee shop problem, Figure \ref{figure1} appears distant from providing an accurate solution. As we consider the previous location of the sequence at each step, the configuration seems far from well distributed for $N=3$ or $N=1$ (Figure \ref{figure2}). Figure \ref{figure3} seems to provide a better solution to the coffee shop problem.

\color{black}
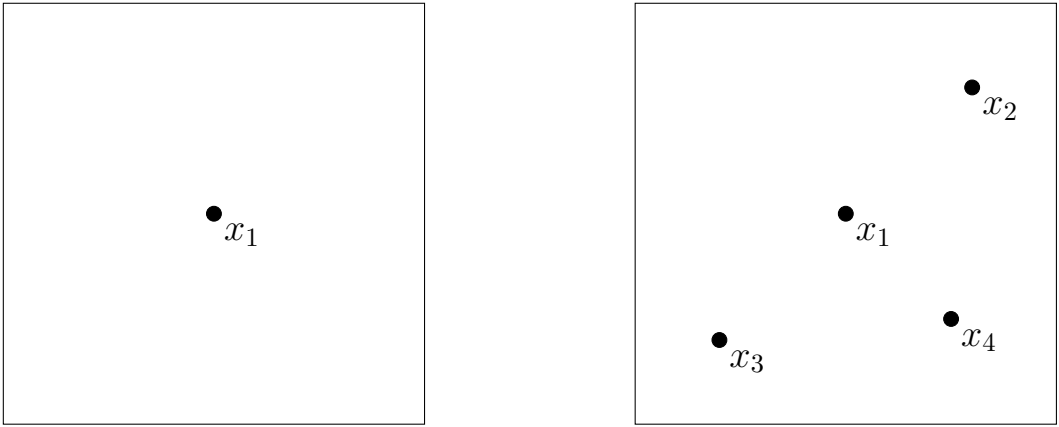
\begin{figure}[ht]
\centering
\begin{tikzpicture}[scale=2.8]
  \draw (0,0) rectangle (2,2);
  \draw (3,0) rectangle (5,2);
  \filldraw (1,1) circle (1pt) node [below right] {\Large $x_1$};
  \filldraw (4,1) circle (1pt) node[below right] {\Large $x_{1}$};
  \filldraw (4.6,1.6) circle (1pt) node[below right] {\Large $x_{2}$};
  \filldraw (3.4,0.4) circle (1pt) node[below right] {\Large $x_{3}$};
  \filldraw (4.5,0.5) circle (1pt) node[below right] {\Large $x_{4}$};
\end{tikzpicture}
\caption{On the left: the settlement of the first Coffee shop. On the right: a settlement for four Coffee shops that seems to fit better the coffee shop problem than Figure \ref{figure1}}
    \label{figure3}
\end{figure} 
\color{black}

A theoretical approach to these kind of questions about uniformly approximating regions by a discrete set of points is the one developed by the \textit{geometric discrepancy theory}. The interested reader can find more about that perspective in this survey \cite{reviewbilyk}. We also present here more references related this interesting research field \cite{blumlinger,slpmathematicalmodel,edward2,edward,edward3}. Also, we want to highlight that, from the probability theory perspective, some work has also been developed \cite{ambrosio,bolleyvillani,talagrand}. 

Although geometric discrepancy theory is a very fruitful area, in order to approach the coffee shop problem, we have decided to deal with it from the \textit{optimal transport} and Wasserstein space perspective. Optimal transport has shown its power to solve many different problems in a vast number of applied scenarios  \cite{santambrogio,villanitopics}.  We formally state the coffee shop problem as the problem of choosing $x_N\in X$ such that the following identity is satisfied:  

\begin{equation}{\label{2wasserstein}}
W_2\left(\frac1N\sum_{i=1}^{N}\delta_{x_i},dx\right)=\min_{x\in X}W_2\left(\frac{1}{N}\sum_{i=1}^{N-1}\delta_{x_i}+\frac1N\delta_{x},dx\right),
\end{equation}
where $W_2$ denotes the $2$--Wasserstein distance, $dx$ is the Riemannian volume measure of our space, normalized with $dx(X)=1$, $\{x_i\}_{i=1}^{N-1}\subset X$ is a finite subset, and $\delta_{x_i}$ denotes the Dirac measure at $x_i$. In this context, the Wasserstein distance models how close a uniformly distributed population is from a set of 
points, particularly, a set of coffee shops.

Louis Brown and Stephan Steinerberger dealt with the problem of distributing points evenly over a compact domain in  \cite{brownsteinerberger2,brownsteinerberger,steinerbergermanifold}. Particularly, they proved bounds on the cost of transporting Dirac measures supported on $\{x_1,\dots,x_N\}\subset X$ to the normalized measure $dx$. These valuable inequalities relate the size of $N$ with how close the points are to the uniform measure.  Brown and Steinerberger's work has assisted us in addressing our modification to the coffee shop problem. Here, we briefly present their results.

In \cite{steinerbergermanifold}, Steinerberger uses the heat kernel and the Green function to obtain the following result for any sequence: 

\begin{thm}[Steinerberger, 
{\cite[Theorem 1]{steinerbergermanifold}}]{\label{steinerbergermanifold}}
Let $X$ be a smooth, compact $d$--dimensional manifold without boundary. Then, for any set of $N$ points $\{x_1,...,x_N\}\subset X$, the following inequality holds\begin{equation}{\label{steinerbergerresult}}
W_2\left(\sum_{i=1}^{N}\frac{1}{N}\delta_{x_i},dx\right)\lesssim_X\frac{1}{N^{1/d}}+\frac1N\left|\sum_{k\neq l}G(x_k,x_l)\right|^{1/2},
\end{equation}where $G:X\times X\to\mathbb{R}\cup\{\infty\}$ denote the Green's function of the Laplacian normalized to have average value $0$ over the manifold and $d\geq3$.

If the manifold is two--dimensional, a slightly different inequality also holds \begin{equation}
W_2\left(\sum_{i=1}^{N}\frac{1}{N}\delta_{x_i},dx\right)\lesssim_X\frac{\sqrt{\log N}}{N^{1/2}}+\frac1N\left|\sum_{k\neq l}G(x_k,x_l)\right|^{1/2}.
\end{equation}
\end{thm}

\begin{rema} 
    The symbol $\lesssim_X$ denotes the same relationship as $\le$, but omits multiplication by a constant on the right-hand side that specifically depends on the manifold $X$. 
    Steinerberger elaborates on this constant in \cite[Section 3]{steinerbergermanifold}. For a more detailed explanation, readers can also consult \cite{aronson,liyyau}.
\end{rema}

Later, Brown and Steinerberger eliminate the Green term in \eqref{steinerbergerresult} by defining the recursive sequence $x_N = \operatorname{arg min}_{x\in X} \sum_{k=1}^{N-1} G(x_k, x)$ \cite[Theorems 1 \& 3]{brownsteinerberger}. This is called a \textit{greedy sequence}. Additionally, in \cite{brownsteinerberger2}, they obtain similar bounds using other sequences, but on the $d$--dimensional torus $\mathbb{T}^d$.

In comparison to the original problem, we introduce competition in the region $X$. This modification seems natural as, in a city, different coffee brands compete for control over certain areas.
In this new scenario, the fundamental concept is to compare a new measure 
\begin{equation}\label{nuevamedida}
\mu=\frac{1}{N_1+N_2}\left[\sum_{i=1}^{N_1}\delta_{x_i}-\sum_{j=1}^{N_2}\delta_{y_j}\right].
\end{equation}
with respect to $dx$. Here, our shops are represented by the positive deltas, while rival shops are represented by the negative ones. A discussion about the choice of the constant multiplying the subtraction of the summations of Dirac deltas has been included in Appendix \ref{appendixone}. However, for the sake of brevity, we believe that $\frac{1}{N_1+N_2}$ is the one that best approximates the real situation we are considering as $N_1+N_2$ are the total number of stores in the region and it behaves as a normalization term in $\mu$.

Now, it is essential to note that $\mu$ is not a probability measure, and, in fact, it is not even positive. To address this, we opt for the \textit{Signed Wasserstein distance} introduced by Piccoli, Rossi, and Tournus in \cite{piccoli} and defined as follows:
\[
\mathbf{W}_1^{1,1} (\mu, \nu) := W_1^{1,1} (\mu_+ + \nu_- , \mu_- + \nu_+),
\]
where $\mu = \mu_+ - \mu_-$ denotes the Jordan decomposition of a signed measure, and $W_1^{1,1}$ is the \textit{generalized Wasserstein distance} between $\mu, \nu\in \mathcal M(X)$ given by
\[W_1^{1,1}(\mu,\nu) = \inf_{\substack{ \tilde{\mu}, \tilde{\nu}\in \mathcal M (X) \\ |\tilde{\mu}| = |\tilde{\nu}|}} |\mu-\tilde{\mu}| + |\nu- \tilde{\nu}| +  W_1(\tilde{\mu}, \tilde{\nu}) .\] 
Recall that $\mathcal M(X)$ is the space of positive finite measures supported in $X$.

The Signed Wasserstein distance is less restrictive than the canonical Wasserstein and accommodates signed finite measures like $\mu$. The monotonicity of the $1$-Wasserstein distance (that is, $W_1(\mu, \nu) \le W_2(\mu, \nu)$ by Hölder inequality), will allow us to apply Steinerberger and Brown's results. Employing the Signed Wasserstein distance, we can now effectively compare $\mu$ with $dx$, obtaining bounds and establishing a robust framework for our new problem.


To clarify our choice of measure $\mu$, let us consider the following: recognizing that competition negatively affects our Coffee Shop brand, we incorporate this factor by subtracting rivals ($\sum_{j=1}^{N_2}\delta_{y_j}$) from our stores in \eqref{2wasserstein}. Equation \eqref{2wasserstein} serves as a metric, measuring how closely the benefit from our Coffee Shop at ${x_i}$ aligns with the benefit derived from uniformly distributing stores across region $X$. Introducing competition requires subtracting the rival's benefit from ours, leading to a signed measure $\mu$. Moreover, in the spirit of the optimal transport problem, with this subtraction we measure how close a uniformly distributed population is form a set of points, regarding the rival brand is negatively affecting us. 


\color{black}

Thus, following the spirit of \eqref{2wasserstein}, the \textit{rival coffee shop problem} can be understood as the problem of choosing $x_N\in X$ such that the following identity is satisfied:

\begin{equation}{\label{newproblem}}
    \mathbf{W}_1^{1,1}\left(\mu,dx\right)=\min_{x\in X}\mathbf{W}_1^{1,1}\left(\frac{1}{N_1+N_2}\left[\sum_{i=1}^{N_1-1}\delta_{x_i}-\sum_{j=1}^{N_2}\delta_{y_j}\right]+\frac{1}{N_1+N_2}\delta_x,dx\right).
\end{equation}

In order to deal with the rival coffee shop problem, we have considered two different scenarios: fixed and dynamic competition.

\color{black}

In the fixed competition scenario, we assume that the number of rivals is settled at $N_2>0$, and we obtain the same bounds as Steinerberger and Brown \cite{brownsteinerberger2, brownsteinerberger,steinerbergermanifold}. First we replicate the bound for any set of points, depending on a Green function: 

\begingroup
\def\thetheorem{\ref{teoremauno}}
\begin{theorem}
Let $X$ be a smooth, compact $d$--dimensional manifold without boundary, $d\geq 3$, $G:X\times X\to\mathbb{R}\cup\{\infty\}$ denote the Green's function of the Laplacian normalized to have average value $0$ over the manifold and $N_1, N_2>0$. Then for any distinct sets of points $\{x_1,\dots,x_{N_1}\}$ and $\{y_1,\dots,y_{N_2}\}$ we obtain
    \[ \mathbf{W}_1^{1,1}(A,dx)
\lesssim_X\frac{2N_2}{N_1+N_2}+\frac{1}{(N_1+N_2)^{1/d}} + \frac{1}{N_1+N_2} \left| \sum_{k\neq \ell} G(z_k, z_\ell) \right| ^{1/2}.
\]
    where \[A = \frac{1}{N_1+N_2}\left[\sum_{i=1}^{N_1}\delta_{x_i}-\sum_{j=1}^{N_2}\delta_{y_j}\right]\] 
    and $z_i=x_i$ from $i=1$ to $N_1$ and $z_i=y_{i-N_1}$ for $i=N_1+1$ to $N_1+N_2$.
\end{theorem}
\addtocounter{theorem}{-1}
\endgroup

We have also achieved the same bound after choosing a \textit{greedy sequence} (the explanation of these sequeces is placed on Section \ref{sectionthree}) to get rid of the Green function factor:

\begingroup
\def\thetheorem{\ref{thm33}}
\begin{theorem}
Let $z_n$ be a greedy sequence on a
$d$--dimensional compact manifold with $d\geq 3$ and $\{x_1,\dots,x_{N_1}\}\subset\{z_i\}_{i=1}^{N_1+N_2}$ and $\{y_1,\dots,y_{N_2}\}\subset\{z_i\}_{i=1}^{N_1+N_2}$ such that $x_i\neq y_j$ for arbitrary $i$, $j$. Then \[
\mathbf{W}_1^{1,1}\left(\frac{1}{N_1+N_2}\left[\sum_{i=1}^{N_1}\delta_{x_i}-\sum_{j=1}^{N_2}\delta_{y_j}\right],dx\right) {\color{black} \lesssim_{X}\frac{2N_2}{N_1+N_2}+\frac{1}{(N_1+N_2)^{1/d}}.}
\]
\end{theorem}
\addtocounter{theorem}{-1}
\endgroup

We also establish a weak lower bound in the same spirit as the one obtained in \cite{brownsteinerberger2,brownsteinerberger,steinerbergermanifold}.

\begin{rema}
    As Steinerberger stated on \cite[page 4]{steinerbergermanifold}, Theorem \ref{steinerbergermanifold} is sharp for $d\geq 3$ and sharp up to possibly the factor of $\sqrt{\log n}$ in $d=2$. Moreover, he gave a brief explanation about that statement. 
    Because of that, we decided to present our results only for the case of $d\geq 3$ as our results should have the same behaviour as theirs for $d=2$.
\end{rema}

In the dynamic scenario, we consider three different cases:

\begin{enumerate}
    \item An area of $X$ is controlled by the rival.
    \item The rival's company grows faster than ours.
    \item The competition has the same growth rate as ours. 
    
\end{enumerate}

In the first case, we impose $A\subset X$ a subset where we cannot settle any Dirac delta. Then, we study in Proposition \ref{prop41} and Corollary \ref{cor42} how the brand whose stores are place around $X$ without restriction will have a winning strategy as they will have less $1$--Wasserstein distance against $dx$ than the other company.


\color{black} In the last two cases, we will work with a rival who is comparable to us. Then, the measure $\mu$ will most likely not approach $dx$ for $N\to \infty$, so it will make sense to simply consider which brand is closer to the uniform distribution. Now, the same way that $\mathbf{W}_1^{1,1}(\mu, dx)$ is a metric of how well distributed our shops are (the lower the better), $\mathbf{W}_1^{1,1}(-\mu, dx)$ measures how well distributed the rival shops are, because the minus sign exchanges the positive and the negative deltas. Thus, we will say that \textit{the rival shops will win} if they have a better position, that is,
\[
\mathbf{W}_1^{1,1}(-\mu, dx) < \mathbf{W}_1^{1,1}(\mu, dx).
\]
In that sense, we have obtained two interesting results:

\color{black}
\begingroup
\def\thetheorem{\ref{proposicioncon2N}}
\begin{theorem}
    {\color{black} Let f: $\mathbb{N} \to \mathbb{N}$ } and $\mu_N= \left(\sum_{i=1}^N \delta_{x_i} -\sum_{j=1}^{f(N)} \delta_{y_j}\right)$. If $f(N) \geq f(N-1) + 2$, then, for $N_0$ big enough, the rival shops will {\color{black} have a winning strategy for all $N\geq N_0$, i.e., they can choose a sequence such that}
    \begin{equation}
\mathbf{W}_1^{1,1} \left(    \frac{1}{N+f(N)}\mu_N,dx \right)  > \mathbf{W}_1^{1,1} \left( \frac{1}{N+f(N)}(-\mu_N), dx \right).
    \end{equation}
\end{theorem}
\addtocounter{theorem}{-1}
\endgroup

\begingroup
\def\thetheorem{\ref{thm45}}
\begin{theorem}
  Let $\mu_N= \left(\sum_{i=1}^N \delta_{x_i} -\sum_{j=1}^{N+K} \delta_{y_j}\right)$, and $N_0>0$. Then, there exist values of $K$ such that the rival shops will have a winning strategy for all $N\le N_0$, i.e. 
    \begin{equation}
\mathbf{W}_1^{1,1} \left(    \frac{1}{2N+K}\mu_N,dx \right)  > \mathbf{W}_1^{1,1} \left( \frac{1}{2N+K}(-\mu_N), dx \right).
        \end{equation}
\end{theorem}
\addtocounter{theorem}{-1}
\endgroup

    

The structure of the paper is as follows: Section \ref{sectiontwo} presents the definitions and notation for the generalization of the Wasserstein distance in order to introduce the Signed Wasserstein distance. In Section \ref{sectionthree}, we study the case where the competition is fixed. In Section \ref{section4}, we analyze the three different scenarios for dynamic competition. Lastly, Appendix \ref{appendixone} provides a brief discussion of the constant that multiplies $\sum_{i=1}^{N_1}\delta_{x_i}-\sum_{j=1}^{N_2}\delta_{y_j}$.

The original impetus for this paper arose from the authors' participation in the PIMS-IFDS-NSF Summer School on Optimal Transport in Seattle in June, 2022 and a fruitful talk with Stephan Steinerberger after his lecture. The authors would like to thank the Kantorovich Initiative for organizing this stimulating event. They also want to thank Jaime Santos for his valuable comments and guidelines at different stages of this paper and their advisor Luis Guijarro for his final appreciations of the manuscript.

\color{black}
\section{Generalizing the Wasserstein distance}{\label{sectiontwo}}
In this section we will recall the standard notions of optimal transport and Wasserstein distance. Then we will expand the definitions to signed measures. The definition of signed measure distance
dates back to Piccoli \cite{piccoli}, but see also Bubenik \cite[Section 1.1.1]{bubenik} for an alternative
exposition.

Let $X$ be a compact Riemannian manifold of dimension $d$ with distance function $\dist$. Denote by $\mathcal M(X)$ the set of positive finite measures in $X$, and by $\mathcal M^s(X)$ the set of \emph{signed} finite measures in $X$. In addition, we will call $dx$ the normalized Riemannian volume measure, (i.e. $dx(X)=1$).

\begin{definition}
A \emph{transference plan} between two positive measures $\mu, \nu \in \mathcal M(X)$ of the same mass is a finite positive measure $\pi \in \mathcal M (X \times X)$ which satisfies that, for all $A, B$ Borel subsets of $X$,
\[ \pi(A\times X) = \mu(A), \quad \text{and} \quad \pi(X \times B) = \nu(B). \]
\end{definition}
In other words, a valid plan $\pi$ has to have $\mu$ and $\nu$ as marginal distributions. Note that we require $|\mu| = | \nu| = \pi( X \times X)$, where $|\mu|$ denote the total mass of a certain measure $\mu$. We denote by $\Gamma(\mu, \nu)$ the set of transference plans between those two measures. Then, we define the $p$-Wasserstein distance for $p\ge 1$ and two positive Radon measures of the same mass by
\[
W_p( \mu, \nu) := \left( \min_{\pi \in \Gamma(\mu, \nu)} \int_{X \times X }\dist(x,y)^p d\pi (x, y) \right) ^{\frac{1}{p}} .
\]

In order to calculate distances in our model, we will need to generalize this distance first for possibly $|\mu| \neq |\nu|$ and then to signed measures. We can do that as follows.
\begin{definition}{\label{definiciongeneralizada1}}[Generalized  Wasserstein distance, \cite{mainini},\cite{piccoli}]
Let $\mu, \nu$ be two positive measures in $\mathcal M (X)$ with possibly different mass. The generalized Wasserstein distance between $\mu$ and $\nu$ is given for $p\ge 1$ $a>0$ and $b>0$ by
\[
W_p^{a,b}(\mu,\nu) = \left( \inf_{\substack{ \tilde{\mu}, \tilde{\nu}\in \mathcal M (X) \\ |\tilde{\mu}| = |\tilde{\nu}|}} a^p( |\mu-\tilde{\mu}| + |\nu- \tilde{\nu}|)^p + b^p W_p^p(\tilde{\mu}, \tilde{\nu}) \right)^{1/p}.
\]
\end{definition}
We can now extend this distance to signed measures $\mu, \nu \in \mathcal M^s(X)$ by decomposing $\mu= \mu_+ - \mu_-$, known as the \textit{Jordan decomposition}. Here, $\mu_+$ and $\mu_-$ are nonnegative finite measures mutually singular with disjoint support. The property of being mutually singular can be understood as the existence of a set $X^+\in\mathcal{B}$, where $\mathcal{B}$ denote the $\sigma$-algebra, such that $\mu_+(X\backslash X^+)=0$ and $\mu_-(X^+)=0$. Of course, the same decomposition exists for $\nu$. 

The total mass of a signed measure $\mu = \mu_+ - \mu_-$ is defined as the sum of the total mass of its parts, that is, $|\mu| := |\mu_+| + |\mu_-| = \mu_+(X) + \mu_-(X)$.

\begin{definition}[Signed Generalized  Wasserstein distance, \cite{mainini},\cite{piccoli}]{\label{definiciongeneralizada2}} Let $\mu, \nu \in \mathcal M^s (X)$. We define their distance by
\[
\mathbf{W}_1^{a,b} (\mu, \nu) := W_1^{a,b} (\mu_+ + \nu_- , \mu_- + \nu_+).
\]
\end{definition}

\begin{rema}
    $W_p^{a,b}$ is not a distance for $p>1$ (see \cite[page 2]{piccoli}). But, if we pick $p=1$, $\mathbf{W}_1^{a,b}$ is indeed a distance (see \cite[Section 3]{piccoli}). In fact, it can be shown that for all choices of $a,b>0$ these distances are induced by equivalent norms $\norm{\mu}^{a,b} := W_1^{a,b}(\mu, 0) = W^{a,b}_1(\mu^+, \mu^-)$ with that equivalence being
    \[
    \min\{a,b\} \norm{\mu}^{1,1} \le \norm{\mu}^{a,b} \le \max\{a,b\} \norm{\mu}^{1,1}.
    \]
    We will work with $\mathbf{W}_1^{1,1}$ on the paper, because we have \[\mathbf{W}_1^{1,1} \le \frac{1} {\min\{a,b\} }\mathbf{W}_1^{a,b}\] as a consequence of the equivalence. 
\end{rema}

The key for $\mathbf{W}^{a,b}_1$ to induce a norm is that {\color{black} the distance does not change after adding any measure at both sides of $\mathbf{W}^{a,b}_1(\cdot, \cdot)$}. Specifically, if $\mu, \nu, \eta \in \mathcal M^s(X)$ are finite signed measures, then it is true that

\[
\mathbf{W}^{a,b}_1(\mu+\eta, \nu + \eta ) = \mathbf{W}^{a,b}_1(\mu, \nu).
\]

This property ensures that $\mathbf{W}^{a,b}_1$ is a distance, since it implies that the distance between two measures is unaffected by adding the same amount of mass to both of them.

\begin{rema}
    As we pointed in the Introduction, although Brown and Steinerberger's results are stated in terms of the $2$--Wasserstein distance, due to the Hölder inequality, we can extend them to the $1$--Wasserstein distance in order to use the Generalized and Signed Wasserstein distance.

    A very interesting open question is regarding to extend their results to $p>2$ as it is not straightforward from the proofs. Moreover, generalize them for geodesic metric spaces would also be of high interest.
\end{rema}

We would also like to mention that there exist other notions of Wasserstein distances that may be applicable to the problem we are considering. One such distance is the \textit{Unbalanced Wasserstein Distance}, introduced by Liero, Mielke, and Savaré in \cite{liero} and by Chizat, Peyré, Schmitzer, and Vialard in \cite{gabrielpeyre}. This distance is a natural way to handle positive measures with different masses, and it benefits from concentrations of mass in the region of the support. However, it is not designed for signed measures, which are necessary for our purposes. Interested readers can find more information about the Unbalanced Wasserstein Distance in the following reference \cite{chizattesis}. In addition, the partial optimal transport presented by Figalli and Gigli \cite{figalli2,figalli3} could also be a good setting for generalize all these kind of optimization, localization and transport problems.

\section{Fixed competition}{\label{sectionthree}}
As we have pointed out in the introduction, we divide our study into two cases. In the first one, the competition only opens a fixed number $N_2>0$ of Coffee Shops. With the distance described in Section \ref{sectiontwo}, we want to see how \[\mathbf{W}_1^{1,1}\left(\frac{1}{N_1+N_2}\left[\sum_{i=1}^{N_1}\delta_{x_i}-\sum_{j=1}^{N_2}\delta_{y_j}\right],dx\right)\] behaves for all $N_1$.

Intuitively, when $N_1$ is much bigger than $N_2$, the rival's influence will be very small. We will formalize that in the following results. 

    \begin{thm}{\label{teoremauno}}
    Let $X$ be a smooth, compact $d$--dimensional manifold without boundary, $d\geq 3$, $G:X\times X\to\mathbb{R}\cup\{\infty\}$ denote the Green's function of the Laplacian normalized to have average value $0$ over the manifold and $N_1, N_2>0$. Then for any distinct sets of points $\{x_1,\dots,x_{N_1}\}$ and $\{y_1,\dots,y_{N_2}\}$ we obtain
    \[ \mathbf{W}_1^{1,1}(A,dx)
\lesssim_X\frac{2N_2}{N_1+N_2}+\frac{1}{(N_1+N_2)^{1/d}} + \frac{1}{N_1+N_2} \left| \sum_{k\neq \ell} G(z_k, z_\ell) \right| ^{1/2}.
\]
    where \[A = \frac{1}{N_1+N_2}\left[\sum_{i=1}^{N_1}\delta_{x_i}-\sum_{j=1}^{N_2}\delta_{y_j}\right]\] and $z_i=x_i$ from $i=1$ to $N_1$ and $z_i=y_{i-N_1}$ for $i=N_1+1$ to $N_1+N_2$.
    \end{thm}
    We note that the Green function is defined by
\[
G(x,y) = \sum_{k=1}^\infty \frac{\phi_k(x) \phi_k(y)}{\lambda_k},
\]
where $\phi_k$ are the eigenfunctions of the Laplace operator on $X$, and $\lambda_k$ its respective eigenvalues. That is, $-\Delta \phi_k = \lambda_k \phi_k$.
    \begin{proof}
    For the sake of simplicity, we denote \[A=\frac{1}{N_1+N_2}\left[\sum_{i=1}^{N_1}\delta_{x_i}-\sum_{j=1}^{N_2}\delta_{y_j}\right]\text{ and }B=\frac{1}{N_1+N_2}\left[\sum_{i=1}^{N_1}\delta_{x_i}+\sum_{j=1}^{N_2}\delta_{y_j}\right].\] We can decompose $A=A_+-A_-$ with $A_+=\frac{1}{N_1+N_2}\sum_{i=1}^{N_1}\delta_{x_i}$ and $A_-=\frac{1}{N_1+N_2}\sum_{j=1}^{N_2}\delta_{y_j}$. Using the definition in \ref{definiciongeneralizada2} we have that \[
    \mathbf{W}_1^{1,1}(A,dx)=W_1^{1,1}(A_+,dx+A_-)=\inf_{\substack{ \tilde{\mu}, \tilde{\nu}\in \mathcal M (X) \\ |\tilde{\mu}| = |\tilde{\nu}|}}\left(|A_+-\widetilde{\mu}|+|dx+A_--\widetilde{\nu}|+W_1(\widetilde{\mu},\widetilde{\nu})\right).
    \]Now, we choose $\widetilde{\mu}=B$ and $\widetilde{\nu}=dx$. So, we obtain, \begin{align*}
    \mathbf{W}_1^{1,1}(A,dx)&\leq\left(|A_+-B|+|dx+A_--dx|+W_1(B,dx)\right)\\
    &=\left|\frac{1}{N_1+N_2}\sum_{j=1}^{N_2}\delta_{y_j}\right|+\left|\frac{1}{N_1+N_2}\sum_{j=1}^{N_2}\delta_{y_j}\right|+W_1(B,dx)\\
    &=\frac{2N_2}{N_1+N_2}+W_1(B,dx).
    \end{align*}
    Now we will combine it with an upper bound for $W_1(B,dx)$ given in \cite[Theorem 1]{steinerbergermanifold}:
    \[
    W_1(B,dx) \le W_2(B, dx) \lesssim_X \frac{1}{(N_1+N_2)^{1/d}} + \frac{1}{N_1+N_2} \left| \sum_{k\neq \ell} G(z_k, z_\ell) \right| ^{1/2}
    \]
    
    Moreover,  \[\frac{2N_2}{N_1+N_2}+\frac{1}{(N_1+N_2)^{1/d}}\leq\frac{2N_2+1}{(N_1+N_2)^{1/d}},\]due to $N_1+N_2>(N_1+N_2)^{1/d}$. Putting everything together, we obtain the desired result: \begin{align*}
        \mathbf{W}_1^{1,1}(A,dx)&\leq\frac{2N_2}{N_1+N_2}+W_1(B,dx)\\
        &\lesssim_X\frac{2N_2}{N_1+N_2}+\frac{1}{(N_1+N_2)^{1/d}} + \frac{1}{N_1+N_2} \left| \sum_{k\neq \ell} G(z_k, z_\ell) \right| ^{1/2}\\
        &\leq\frac{2N_2+1}{(N_1+N_2)^{1/d}}+ \frac{1}{N_1+N_2} \left| \sum_{k\neq \ell} G(z_k, z_\ell) \right| ^{1/2}\\
        &\leq\frac{2N_2+1}{(N_1+N_2)^{1/d}}+ \frac{2N_2+1}{N_1+N_2} \left| \sum_{k\neq \ell} G(z_k, z_\ell) \right| ^{1/2} \qedhere
    \end{align*}
    \end{proof}

Now suppose the sequence $z_n$ is defined in the following way:
\begin{equation}{\label{greedy}}
z_n = \arg \min_x \sum_{k=1}^{n-1} G(x, x_k)
\end{equation}
We will say that such sequence is a \emph{greedy sequence} or that it is \emph{defined in a greedy manner}. 

\begin{thm}{\label{thm33}}
Let $z_n$ be a sequence obtained in the previous way on a $d$--dimensional compact Riemannian manifold with $d\geq 3$ and $\{x_1,\dots,x_{N_1}\}\subset\{z_i\}_{i=1}^{N_1+N_2}$ and $\{y_1,\dots,y_{N_2}\}\subset\{z_i\}_{i=1}^{N_1+N_2}$ such that $x_i\neq y_j$ for all $i$, $j$. Then \[
\mathbf{W}_1^{1,1}\left(\frac{1}{N_1+N_2}\left[\sum_{i=1}^{N_1}\delta_{x_i}-\sum_{j=1}^{N_2}\delta_{y_j}\right],dx\right) {\color{black} \lesssim_{X}\frac{2N_2}{N_1+N_2}+\frac{1}{(N_1+N_2)^{1/d}}.}
\]
\end{thm}

\begin{proof}
We will pick the same notation as in the proof of Theorem \ref{teoremauno}. In this case, we are going to use another result from Steinerberger together with Brown \cite[Theorem 3]{brownsteinerberger}, which entails that for a sequence $z_n$ constructed in a greedy way, we have that \begin{equation} \label{cotasteinerberger}
W_2\left(\frac1n\sum_{k=1}^n\delta_{z_k},dx\right)\lesssim_X\frac{1}{n^{1/d}},
\end{equation}
for $d\geq 3$. So, in our case, 
\begin{align*}
    \mathbf{W}_1^{1,1}(A,dx)&\leq\frac{2N_2}{N_1+N_2}+W_1(B,dx)\\
    &\leq\frac{2N_2}{N_1+N_2}+\frac{1}{(N_1+N_2)^{1/d}},
\end{align*}and we obtain our result.
\end{proof}

{\color{black}In \cite[Section 1.2]{brownsteinerberger2}, Brown and Steinerberger argue that their bound \eqref{cotasteinerberger} is the best in terms of $N$ possible, because, for every set of points $\{x_1, \ldots, x_N\} \subset X$ we have
\begin{equation} \label{cotasteinerbergerinversa}
W_1\left( \frac 1 N \sum_{i=1}^N \delta_{x_i}, dx \right) \ge \frac{c}{N^{1/d}} 
\end{equation}
for $c>0$ a constant depending only on the manifold.
}
\begin{prop}
We have a lower bound that is independent of the sets $\{x_1, \ldots, x_{N_1} \}$ and $\{y_1, \ldots, y_{N_2}\}$. Indeed,
\[
\mathbf{W}_1^{1,1}\left(\frac{1}{N_1+N_2}\left[\sum_{i=1}^{N_1} \delta_{x_i} - \sum_{j=1}^{N_2} \delta_{y_j} \right], dx\right) \ge \frac{c}{(N_1+N_2)^{1/d}} - \frac{2N_2}{N_1+N_2},
\]
where $c>0$ {\color{black} is the same constant from \eqref{cotasteinerbergerinversa} } that depends only on the manifold $X$.
\end{prop}
If $N_2$ is fixed and $N_1 \to \infty$, {\color{black} this bound is asymptotically as good as \eqref{cotasteinerbergerinversa} when $N\to \infty$, because they are comparable for big values of $N$:
\[
\lim_{N\to \infty} \frac{\frac{c}{(N+N_2)^{1/d}} - \frac{2N_2}{N+N_2}}{ \frac{c}{N^{1/d}}} = 1
\]}
\begin{proof}
 By the triangle inequality we know that

\[ 
    W_1\left(\frac{1}{N_1+N_2}\left[\sum_{i=1}^{N_1} \delta_{x_i} + \sum_{j=1}^{N_2} \delta_{y_j} \right], dx\right)  \] 
    \begin{equation} \label{cuentaprop33} \le \mathbf{W}_1^{1,1}\left(\frac{1}{N_1+N_2}\left[\sum_{i=1}^{N_1} \delta_{x_i} - \sum_{j=1}^{N_2} \delta_{y_j} \right], dx\right) \end{equation}
    \[+ \mathbf{W}_1^{1,1}\left(\frac{1}{N_1+N_2}\left[\sum_{i=1}^{N_1} \delta_{x_i} + \sum_{j=1}^{N_2} \delta_{y_j} \right], \frac{1}{N_1+N_2}\left[\sum_{i=1}^{N_1} \delta_{x_i} - \sum_{j=1}^{N_2} \delta_{y_j} \right] \right) .\]
    As $\mathbf{W}_1^{1,1}$ is invariant by ations, we can bound the last term by
    
    \begin{equation}\label{cotaprop34} \begin{split}\mathbf{W}_1^{1,1}\left(\frac{1}{N_1+N_2}\left[\sum_{i=1}^{N_1} \delta_{x_i} + \sum_{j=1}^{N_2} \delta_{y_j} \right], \frac{1}{N_1+N_2}\left[\sum_{i=1}^{N_1} \delta_{x_i} - \sum_{j=1}^{N_2} \delta_{y_j} \right] \right)  \\
    =
W_1^{1,1} \left(0, \frac{2}{N_1+N_2} \sum_{j=1}^{N_2} \delta_{y_j} \right) \le  \left| \frac{2}{N_1+N_2} \sum_{j=1}^{N_2} \delta_j \right| = \frac{2N_2}{N_1 + N_2}. \end{split}
\end{equation}
The inequality is obtained by choosing $\tilde{\mu} = \tilde{\nu} = 0$ in the infimum inside of $W_1^{1,1}$.

Now, {\color{black}  substituting \eqref{cotasteinerbergerinversa} at the left hand side of \eqref{cuentaprop33} and \eqref{cotaprop34} at the right hand side}, we have that
\[
\frac{c}{(N_1+N_2)^{1/d}} \le \mathbf{W}_1^{1,1}\left(\frac{1}{N_1+N_2}\left[\sum_{i=1}^{N_1} \delta_{x_i} - \sum_{j=1}^{N_2} \delta_{y_j} \right], dx\right) + \frac{2N_2}{N_1+N_2},
\]
which clearly implies our result.
\end{proof}

\section{Non--fixed competition}{\label{section4}}


In this section, we provide an overview of scenarios where the rival's growth rate is comparable to ours. Although there may be a general framework that captures all such cases, we choose to examine each scenario separately for clarity.

\subsection{Forbidden areas}

Up until this point, we have measured victory solely in terms of the Signed Wasserstein distance between the difference of the sums of the Dirac deltas and the uniform distribution. An alternative approach is to compute the distance between each individual set of coffee shops and the uniform distribution, as described in the Steinerberger and Brown papers.

In this section, we consider a scenario where our rival has already opened coffee shops and ``colonized" a certain area, such that we are unable to open our own shops within that region of our space $X$. Consequently, our limit as we approach the $dx$ measure will not encompass this region, whereas our rival's will.

The key to our proof lies in the following proposition:

\begin{prop}{\label{prop41}}
Suppose $\{x_i\}$ is any sequence in $X\backslash B_r(p)$, where $B_r(p)$ is an open ball of center $p\in X$ and radius $r>0$. Then,
\[ W_1\left( \frac{1}{N} \sum_{i=1}^N \delta_{x_i} , dx \right)\ge \frac{r}{2} \vol(B_{r/2}(p) ) ,
\]
and in particular $\frac{1}{N} \sum_{i=1}^N \delta_{x_i}$ does not converge to $dx$.
\end{prop}

\begin{proof}
 Suppose $\gamma$ is an optimal transport plan from $ \frac{1}{N} \sum_{i=1}^N \delta_{x_i}$ to the normalized volume measure $dx$. By our hypothesis, there is a mass $\vol(B_{r/2}(p) )$ outside $B_r(p)$ that has to travel a bigger distance than $r/2$ to arrive to $B_{r/2}(p) $. We can then bound below the integral of the definition of the Wasserstein distance by $\frac{r}{2} \vol(B_{r/2}(p) )$, the distance times the volume:
 \begin{align*}
 \int_{X\times X} \dist(x,y) \, d\gamma(x,y) &\ge \int_{X\setminus B_r(p) \times B_{r/2}(p)} (r/2) \, d\gamma(x,y)\\ 
 &= \int_{B_{r/2}(p)} (r/2) \, dy = (r/2) \vol(B_{r/2}(p)). 
 \end{align*}
 In the first inequality we restrict the domain of the integral, so we can bound below the distance. Then we just use Fubini's theorem and the fact that $\gamma$ is a transport plan to obtain the volume measure $dy$ after integrating in $x$.
\end{proof}

By using this, we conclude with a straightforward corollary:

\begin{cor}{\label{cor42}}
Suppose $x_i$ follows a greedy sequence, and $y_j$ is any sequence omitting an open ball. Then, there exist an $N_0$ such that, for every $N\ge N_0$,
\[
W_1\left( \frac{1}{N} \sum_{i=1}^N \delta_{x_i} , dx \right) < W_1\left( \frac{1}{N} \sum_{i=1}^N \delta_{y_j} , dx \right).
\]
\end{cor}
In other words, a smart choice of shops will have better results than any sequence that omits a certain region.

\begin{rema}
Notice that in Section \ref{sectiontwo}, if our rival closes a certain area with their finite coffee shops, we will still win (approaching to $0$ the distance against $dx$ as we open more stores) because their approximation to ${d}x$ will be worse than ours. For that reason, it is important that the rival experiences some growth during the competition.
\end{rema}

\subsection{Rival growth in terms of ours}
\mbox{}%

\medskip



We can express the number of rival coffee shops, $N_2$, in terms of our own, $N_1$, by defining a function $f\colon\mathbb{N}\to\mathbb{N}$ such that $f(N_1)=N_2$. This allows us to summarize many specific cases into a single framework.

In this subsection, we will establish conditions that $f$ must satisfy in order for the rival to defeat us. \color{black} We recall that the winning strategy would be the one with less Signed Wasserstein distance against $dx$. That is, if for the same $\mu$ defined in \eqref{nuevamedida} we have
\[
\mathbf{W}_1^{1,1}(-\mu, dx) < \mathbf{W}_1^{1,1}(\mu, dx),
\]
then we say that the rivals would win because $\mathbf{W}_1^{1,1}(\mu, dx)$ is a metric of how well distributed our deltas are, and $\mathbf{W}_1^{1,1}(-\mu, dx)$ measures how well distributed the rival deltas are (note that the minus sign exchanges the positive and the negative deltas).
\color{black}
We divide this part into two subsections. Before presenting these conditions, we prove a technical lemma that will be used throughout the rest of this section.

\begin{lem}{\label{cotainferiorwassers}}
    Let $\mu, \nu \in \mathcal M^s(X)$ be two signed finite measures. Then,
    \[
    \mathbf{W}^{1,1}_1(\mu, \nu) \ge  | \mu(X) - \nu(X)|.
    \]
\end{lem}
\begin{proof}
It suffices to check the result for positive measures because
\[
|\mu(X) - \nu(X) | = | (\mu^+(X) + \nu^-(X) ) - (\nu^+(X) + \mu^-(X) ) |.
\]
Now, for $\mu, \nu \in \mathcal M(X)$,
 \begin{align*}
W^{1,1}_1 (\mu, \nu) &= \inf_{|\tilde{\mu}|= |\tilde{\nu}|} \left(|\mu - \tilde{\mu}| + |\nu - \tilde{\nu}| + W_1(\tilde{\mu}, \tilde{\nu})\right) \\
&\ge \inf_{|\tilde{\mu}|= |\tilde{\nu}|} \left(|\mu - \tilde{\mu}| + |\nu - \tilde{\nu}|\right) 
\\ &\ge \inf_{|\tilde{\mu}|= |\tilde{\nu}|} 
\left(|\mu - \tilde{\mu} -\nu + \tilde{\nu}|\right) \\&\ge \inf_{\tilde{\mu}(X)= \tilde{\nu}(X)} \left|
\left(\mu - \tilde{\mu} -\nu + \tilde{\nu}\right)(X) \right| = |\mu(X) - \nu(X) |. \qedhere
       \end{align*}




\end{proof}

\subsubsection{Case $f(N)\geq f(N-1)+2$}{\label{421}}
\mbox{}%

\medskip
We present a first result for the dynamic case under the hypothesis of the rival coffee shop complex growing a lot faster than ours. Precisely, we will suppose that $f(N) \ge f(N-1) + 2$. That is, whenever we place a shop, our rivals will place two or more.

\begin{thm}{\label{proposicioncon2N}}
    {\color{black} Let f: $\mathbb{N} \to \mathbb{N}$ } and $\mu_N= \left(\sum_{i=1}^N \delta_{x_i} -\sum_{j=1}^{f(N)} \delta_{y_j}\right)$. If $f(N) \geq f(N-1) + 2$, then, for $N_0$ big enough, the rival shops will {\color{black} have a winning strategy for all $N\geq N_0$, i.e., they can choose a sequence such that}
    \begin{equation}
\mathbf{W}_1^{1,1} \left(    \frac{1}{N+f(N)}\mu_N,dx \right)  > \mathbf{W}_1^{1,1} \left( \frac{1}{N+f(N)}(-\mu_N), dx \right).
    \end{equation}
\end{thm}
\
\begin{proof}
  The idea behind the proof  is that our rival has a winning strategy. That is, at least he is able to copy the placement of our $N$-th shop with one of his shops because in every turn he adds at least two shops by hypothesis. So, he can copy ours and settle other shops in the remaining non-occupied space in $X$. We present a formal computation of this explanation.

  Suppose our sequence of shops is given by $x_1, \ldots, x_N$. Then, following this procedure we would have \begin{equation}\label{sucesionteorema44}
  \begin{split}
      y_1 &= x_1, \\
  y_{f(1) + 1} &= x_2, 
  \\ \vdots& \\ y_{f(N-1)+1} &= x_N, \end{split}
  \end{equation} with every other $y_j$ filling the space in a greedy manner.

We would like to clarify two hidden implications before presenting the final inequalities. Firstly, our hypothesis clearly implies that $f(N) \ge 2N$. The other one is that we will call $J$ the set of indexes of $y_j$ that fills $X$, that is, the ones that do not copy the sequence $x_i$. It is a straightforward computation that the cardinality of $J$ is $|J| = f(N) - N$.

  Now, we are ready to finish our proof. On the one hand, choosing $\tilde{\mu} = \frac{1}{f(N)-N} \sum_{j\in J} \delta_{y_{j}}$ and $\tilde{\nu}=dx$ gives us
  \begin{align*}
        \mathbf{W}_1^{1,1} \left(  \frac{1}{N+f(N)} (- \mu),dx \right) &= \inf_{|\tilde{\mu}| = |\tilde{\nu}|} \left(\left|\frac{1}{N+f(N)} (- \mu) - \tilde{\mu} \right| + |dx - \tilde{\nu}| + W_1 \left(\tilde{\mu}, \tilde{\nu} \right)\right) \\
        &\le
        \left| \left( \frac{1}{f(N)+N} - \frac{1}{f(N)-N} \right) \sum_{j\in J} \delta_{y_{j}} \right|+ W_1 \left( \frac{1}{f(N)-N} \sum_{j\in J} \delta_{y_{j}}, dx \right) \\
        &\le
        \frac{2N}{f(N)+N} + \frac{c}{N^{d}}\le \frac{2}{3} +\frac{c}{N^d},
  \end{align*}due to Brown and Steinerberger results \cite[Theorems 1 \& 3]{brownsteinerberger}. On the other hand, using Lemma \ref{cotainferiorwassers}, 
\[
 \mathbf{W}_1^{1,1} \left(  \frac{1}{N+f(N)}  \mu,dx \right) \ge  1+ \frac{f(N) - N}{f(N)+N}=
 \frac{2f(N)}{N+ f(N)}
 \ge 1.
\]

Finally, we observe that for any manifold $X$ of dimension $d$ we can choose $N_0$ such that $\frac{c}{N^d} < \frac{1}{3}$ for all $N\ge N_0$ (we remind that $c>0$ is a constant depending only on the manifold $X$). Then, for $N\ge N_0$, we conclude  that 
\[
\mathbf{W}_1^{1,1} \left(  \frac{1}{N+f(N)}  (-\mu),dx \right) \le \frac{2}{3}+\frac{c}{N^d} < 1 \le \mathbf{W}_1^{1,1} \left(  \frac{1}{N+f(N)}  \mu,dx \right). \qedhere
\]
\end{proof}

  We imposed $f(N) \ge f(N-1) + 2$ for two reasons: on the one hand, this restriction implies $f(N) \ge 2N$. On the other hand, due to this we were able to explicitly define the winning sequence \eqref{sucesionteorema44}. Losing that clarity we can make a more general statement:
  
\begin{cor}
    The result also holds if $\liminf \frac{f(N)}{N} = \lambda > 1.$
\end{cor}
\begin{proof}
    The assumption $\liminf \frac{f(N)}{N} = \lambda$ means that $f(N)$ will increase in a comparable way to $\lambda N$, so the rival shops will still be able to copy our locations and establish new ones in an optimal way (possibly at a slower rate that once every turn, if $\lambda < 2$).

    The bounds
    \[
    \mathbf{W}_1^{1,1} \left( \frac{1}{N+f(N)} (-\mu) , dx \right) \le \frac{2N}{f(N) + N} + \frac{c}{N^d}
    \]
    and
    \[
    \mathbf{W}_1^{1,1} \left( \frac{1}{N+f(N)} \mu , dx \right)  \ge 1 + \frac{f(N) - N}{f(N) + N}
    \]
    from the proof of Theorem \ref{proposicioncon2N} still applies. Bounding $f(N)\ge \lambda N - \varepsilon$ for $N$ big enough yields the result.
\end{proof}

\subsubsection{Case $f(N)= N+K$} 
\mbox{}%

\medskip

Now, we will suppose $f(N)= N + K$. That is, the rivals will set one shop every time we do, but they start with an advantage. In certain sense, we are growing at the same speed. It seems clear that, for big values of $N$ the two chains of shops will be in a very similar position. For that reason, we will study the situation for fixed $N$.

\begin{thm}{\label{thm45}}
  Let $\mu_N= \left(\sum_{i=1}^N \delta_{x_i} -\sum_{j=1}^{N+K} \delta_{y_j}\right)$, and $N_0>0$. Then, there exist values of $K$ such that the rival shops will have a winning strategy for all $N\le N_0$, i.e. 
    \begin{equation}
\mathbf{W}_1^{1,1} \left(    \frac{1}{2N+K}\mu_N,dx \right)  > \mathbf{W}_1^{1,1} \left( \frac{1}{2N+K}(-\mu_N), dx \right).
        \end{equation}
\end{thm}

\begin{proof}
     In the same spirit as in Subsection \ref{421}, the rival has, at least, the strategy of choosing their first $K$ shops in a greedy manner, and then $y_{n+K} = x_n$ for all $n$. In that case, $\mu _N = -\sum_{j=1}^K \delta_{y_j}$.

    Now, using lemma \ref{cotainferiorwassers}
    we can see that
    \[
\mathbf{W}_1^{1,1} \left(    \frac{1}{2N+K}\mu_N,dx \right) \ge 1+ \frac{K}{2N+K}.
    \]
On the other hand, if we choose $\tilde{\mu} = \frac{1}{K} \sum_{j=1}^K \delta_{y_{j}}$ and $\tilde{\nu}=dx$, we obtain that
  
  \begin{align*}
        \mathbf{W}_1^{1,1} \left(  \frac{1}{2N+K} (- \mu),dx \right) &= \inf_{|\tilde{\mu}| = |\tilde{\nu}|} \left(|\mu - \tilde{\mu} | + |\nu - \tilde{\nu}| + W_1 \left(\tilde{\mu}, \tilde{\nu} \right)\right) \\
        &\le
        \left| \left( \frac{1}{2N+K} - \frac{1}{K} \right) \sum_{j\in J} \delta_{y_{j}} \right|+ W_1 \left( \frac{1}{K} \sum_{j= 1}^K \delta_{y_{j}}, dx \right) \\
        &\le
        \frac{2N}{2N+K} + \frac{c}{K^{d}} = 1 - \frac{K}{2N+K} + \frac{c}{K^d}.
  \end{align*}
In the last inequality we've applied the result of Brown and Steinerberger \cite{brownsteinerberger2} to the greedy sequence $y_1, \ldots, y_K$. We recall that $c$ is a positive constant which depends only on the manifold $X$.

Combining both inequalities, we have shown that whenever 
\begin{equation}\label{requisitoCasoK}
\frac{c}{K^d} \le \frac{2N}{2N+K},
\end{equation}
our result holds. And, by basic calculus, we know that for a fixed $N_0>0$ there exist a number $K_0>0$ such that for any $K\ge K_0$ and all $N\le N_0$, the inequality \ref{requisitoCasoK} is verified. 
\end{proof}

\appendix
\section{Discussion about the constants in the difference of Dirac deltas}{\label{appendixone}}

During the initial steps of this paper, the authors considered three different constants to multiply the difference of the Dirac deltas $\sum_{i=1}^{N_1}\delta_{x_i}-\sum_{j=1}^{N_2}\delta_{y_j}$. In this appendix, we present our considerations about the matter:

\medskip

\begin{enumerate}

\item \textbf{$\frac{1}{N_1}$ in the first term and $\frac{1}{N_2}$ in the second term}: 

\medskip

The main objection to this choice is that it gives different masses to the coffee shops of each team if $N_1\neq N_2$. From our perspective, this does not capture the essence of our problem, as we consider that all Coffee Shops (regardless of which team they belong to) have the same power of attraction and, formally, the same weight.

It would be interesting to consider the problem with different weights. For example, one of the teams could be a big consolidated coffee shop chain while the other team is composed of small ones. For that setting, this constant choice could possibly be appropriate.

\bigskip

\item \textbf{$\frac{1}{N_1-N_2}$ multiplying both factors}:

\medskip

The virtue of this constant is that it normalizes the measure and turns it into a probability measure. In addition, it gives the same weight to each Coffee Shop. It seems that the fixed case computations of this paper hold for this constant. The problem with this choice is the case $N_1=N_2$, which leads to division by zero. Therefore, this constant is not suitable for our problem.

\bigskip

\item \textbf{$\frac{1}{N_1+N_2}$ multiplying both factors}:

\medskip

This constant gives the correct weight to each Coffee Shop regardless of the team they belong to. It appeared when we tried to compute the optimality of all Coffee Shops against the volume measure regardless of which team they belong to. Moreover, as the denominator is always positive, we can use it for every $N_1$, $N_2>0$ and if we join the masses of the two companies it would result in $N_1+N_2$, i.e., the total population of stores. The combination of deltas may not be normalized, but that is not a problem after the generalization of the Wasserstein distance.

\end{enumerate}

    \nocite{*}
    \printbibliography

 \end{document}